\documentclass[12pt]{article}

\usepackage{amsmath}
\usepackage{amssymb}
\usepackage[mathscr]{eucal}
\usepackage{graphicx}
\usepackage{color}

\usepackage{bm}
\usepackage{mathrsfs}
\usepackage{amsthm}
\usepackage{enumerate}
\usepackage[mathscr]{eucal}
\usepackage{eqlist}
\usepackage{graphicx}
\allowdisplaybreaks
\numberwithin{equation}{section}

\newtheorem{Theorem}{\sc Theorem}
\newtheorem{Definition}[Theorem]{\sc Definition}
\newtheorem{Proposition}[Theorem]{\sc Proposition}

\newtheorem{Remark}[Theorem]{\sc Remark}
\newtheorem{Example}[Theorem]{\sc Example}
\newtheorem{Problem}[Theorem]{\sc Problem}

\newcommand{\R}{{\if mm {\rm I}\mkern -3mu{\rm R}\else \leavevmode
        \hbox{I}\kern -.17em\hbox{R} \fi}}

\newcommand{\wto}{ \ \stackrel{w} {\longrightarrow} \ }
\def\sqr#1#2{{
        \vcenter{
            \vbox{\hrule height.#2pt
                \hbox{\vrule width.#2pt height#1pt \kern#1pt
                    \vrule width.#2pt
                }
                \hrule height.#2pt
            }
        }
}}

\def\div{\mathop{\rm div}\nolimits}

\def\real{\mathbb{R}}

\def\lista#1
{{ \itemindent 0.0cm \labelsep .2cm \leftmargin 0.8cm \rightmargin
        0.0cm \labelwidth 0.6cm \topsep 0.0mm
        \parsep 0.0mm
        \itemsep 0.0mm
        \begin{list}{}
            { \setlength{\leftmargin}{.8cm} \setlength{\rightmargin}{0.0cm}
                \setlength{\parsep}{0.0mm} \setlength{\topsep}{.0mm}
                \setlength{\parskip}{.0cm} \setlength{\itemsep}{.0cm} }
            {#1}\end{list}} }

\begin{document}
\title{
A class of elliptic mixed boundary value problems with $(p,q)$-Laplacian: existence, comparison and optimal control 
}

\author{
Shengda Zeng \footnote{\,Guangxi Colleges and Universities Key Laboratory of Complex System Optimization and Big Data Processing,   Yulin Normal University, Yulin 537000, P.R. China, Department of Mathematics, Nanjing University, Nanjing, Jiangsu, 210093, P.R. China,  and Jagiellonian University in Krakow, Faculty of Mathematics and Computer Science, ul. Lojasiewicza~6, 30348 Krakow, Poland. E-mail address: zengshengda@163.com.
Tel.: +86-18059034172.},
\ \ Stanis{\l}aw Mig\'orski \footnote{\, College of Applied Mathematics, Chengdu University of Information Technology, Chengdu, 610225 Sichuan Province, P.R. China, and Jagiellonian University in Krakow, Chair of Optimization and Control, ul. Lojasiewicza 6, 30348 Krakow, Poland.
E-mail address: stanislaw.migorski@uj.edu.pl.},
\ \ Domingo A. Tarzia \footnote{\,
Depto. Matem\'atica, FCE, Universidad Austral,
Paraguay 1950, S2000FZF Rosario, Argentina, and CONICET, Argentina.  E-mail address: DTarzia@austral.edu.ar
},\\
Lang Zou \footnote{\,
School of Mathematics and Computational Science, Xiangtan University
Xiangtan City 210000, Hunan Province, P. R. China.  Corresponding author.  E-mail address: langzou@xtu.edu.cn
}\ \
 and \ \ Van Thien Nguyen \footnote{\,
		Department of Mathematics, FPT University,
Education zone, Hoa Lac High Tech Park, Km29 Thang Long Highway, Thach That ward,
Hanoi, Vietnam. E-mail address: thiennv.k4@gmail.com}
}

\date{}

\maketitle

\noindent {\bf Abstract.} \
The paper deals with two nonlinear elliptic equations with $(p,q)$-Laplacian
and the Dirichlet-Neumann-Dirichlet (DND) boundary conditions, and Dirich\-let-Neu\-mann-Neumann (DNN) boundary conditions, respectively.
Under mild hypotheses, we prove the unique weak  solvability of the elliptic mixed boundary value problems. Then, a comparison and a monotonicity results for the solutions of elliptic mixed boundary value problems are established.
We examine a convergence result which shows that the solution of (DND) can be approached by the solution of (DNN). Moreover, two optimal control problems governed
by (DND) and (DNN), respectively, are considered,
and an existence result for optimal control problems is obtained. Finally, we provide a result on asymptotic behavior of optimal controls and system states, when a parameter
tends to infinity.

\smallskip


\smallskip


\section{Introduction}\label{Section1}

Let $\Omega$ be a bounded domain in $\real^N$ ($N\ge 2$)
with a Lipschitz boundary $\Gamma:=\partial \Omega$ which is divided into three measurable
and mutually disjoint parts $\Gamma_1$, $\Gamma_2$, and $\Gamma_3$ such that
$\Gamma_1$ is of positive measure.
Let $1< q<p<+\infty$, $\alpha$, $\beta$, $\mu>0$,  $b>0$ and $\theta<p^*$,
where $p^*$ is the critical exponent to $p$
(see (\ref{defp1*}) in Section~\ref{Section2}).
Given functions
$g\colon \Omega\to\real$,
$r\colon \Gamma_2\to \real$ and
$l\colon \Gamma_3\times \real\to \real$,
in the paper we are interested in
the investigation of the following mixed boundary value problems involving $(p,q)$-Laplacian operator:
\begin{Problem}\label{problemss1}
Find $u\colon \Omega \to \real$ such that
\begin{align*}
&-\Delta_p u(x)-\mu\Delta_qu(x)
+\beta|u(x)|^{\theta-2}u(x)=g(x)
&\mbox{\rm in} \ & \ \ \Omega, \\[1mm]
&u=0
&\mbox{\rm on} \ & \ \ \Gamma_1,\\
&
-\frac{\partial_{(p,q)} u}{\partial\nu}=r(x)
&\mbox{\rm on} \ & \ \ \Gamma_2,\\
&u=b
&\mbox{\rm on} \ & \ \ \Gamma_3,
\end{align*}
\end{Problem}
\noindent
and
\begin{Problem}\label{problemss2}
Find $u\colon \Omega \to \real$ such that
\begin{align*}
&-\Delta_p u(x)-\mu\Delta_qu(x)+\beta|u(x)|^{\theta-2}u(x)=g(x)
&\mbox{\rm in} \ & \ \ \Omega,\\
&u=0
&\mbox{\rm on} \ & \ \ \Gamma_1,\\
&
	-\frac{\partial_{(p,q)} u}{\partial\nu}=r(x)
&\mbox{\rm on} \ & \ \ \Gamma_2,\\
&
-\frac{\partial_{(p,q)} u}{\partial\nu}
=\alpha \, l(x,u(x))
&\mbox{\rm on} \ & \ \ \Gamma_3,
\end{align*}
where $\Delta_{p}$ denotes the $p$-Laplace differential operator of the form
\begin{equation*}
\Delta_{p}u=\div (|\nabla u|^{p-2}\nabla u)
\ \ \mbox{\rm for all} \ \ u\in W^{1,p}(\Omega),
\end{equation*}
$\nu$ is the outward unit normal at the boundary $\Gamma$, and
\begin{equation*}
\frac{\partial_{(p,q)} u}{\partial \nu}
:=\left(|\nabla u|^{p-2}\nabla u+\mu |\nabla u|^{q-2}\nabla u, \nu \right)_{\real^N}.
\end{equation*}
\end{Problem}

The weak solutions to Problems~\ref{problemss1} and~\ref{problemss2} are understood as follows.
\begin{Definition}\label{DEF1}
We say that

\smallskip

\noindent
{\rm(i)} a function $u\colon \Omega\to \real$ is a weak solution of Problem~$\ref{problemss1}$, if $u\in W^{1,p}(\Omega)$ is such that $u=0$ on $\Gamma_1$, $u=b$ on $\Gamma_3$ and
\begin{align*}
&\int_\Omega\left(|\nabla u(x)|^{p-2}\nabla u(x)+\mu|\nabla u(x)|^{q-2}\nabla u(x),\nabla v(x)\right)_{\real^N}\,dx\\
&\quad +\int_\Omega\beta|u(x)|^{\theta-2}u(x)v(x)\,dx
=\int_\Omega g(x)v(x)\,dx-\int_{\Gamma_2}r(x)v(x)\,d\Gamma
\end{align*}
for all $v\in W^{1,p}(\Omega)$ with $v=0$ on $\Gamma_1\cup\Gamma_3$,
		
\smallskip
	
\noindent
{\rm(ii)} a function $u\colon \Omega\to \real$ is a weak solution of Problem~$\ref{problemss2}$, if $u\in W^{1,p}(\Omega)$ satisfies $u=0$ on $\Gamma_1$ and
\begin{align*}
&\int_\Omega\left(|\nabla u(x)|^{p-2}\nabla u(x)+\mu|\nabla u(x)|^{q-2}\nabla u(x),\nabla v(x)\right)_{\real^N}\,dx+\int_\Omega\beta|u(x)|^{\theta-2}u(x)v(x)\,dx\\
&\quad
+\alpha\int_{\Gamma_3}l(x,u(x))v(x)\,d\Gamma = \int_\Omega g(x)v(x)\,dx-\int_{\Gamma_2}r(x)v(x)\,d\Gamma
\end{align*}
for all $v\in W^{1,p}(\Omega)$ with $v=0$ on $\Gamma_1$.
\end{Definition}

The main feature of our research contains two perspectives.
First, we deal with problems with mixed boundary value conditions and $(p,q)$-Laplacian operator.
Note that $(p,q)$-Laplace operator with $1<q<p$ is  the sum of a $p$-Laplacian and a $q$-Laplacian,
so the energy functional $I(u)$ corresponding to the $(p,q)$-Laplace operator defined by
\begin{equation*}
I(u):=\int_\Omega
\left(
\frac{\|\nabla u\|^p}{p}
+\frac{\|\nabla u\|^q}{q} \right) \,dx
\ \ \mbox{for all} \ \ u\in W^{1,p}(\Omega),
\end{equation*}
is mainly controlled by the exponent $q$
if $u\in B_1(0):=\{u\in W^{1,p}(\Omega)\mid
\|\nabla u\|_{L^p(\Omega;\real^N)}\le 1\}$,
or by the exponent $p$ when $u\in W^{1,p}(\Omega)\setminus B_1(0)$.
This structure impels the huge potential applications
of $(p,q)$-Laplacian in diverse fields, for instance,
it can used to describe exactly the geometry of composites made of two different materials with distinct power hardening exponents. Second perspective concerns
applications in which mixed boundary value problems
are a powerful mathematical tool.
They have been widely applied to explain various complicated natural phenomena and to solve a lot of engineering problems, for instance, contact mechanics problems, semipermeablity problems, and free boundary problems.
The research of mixed boundary value problems with or without $(p,q)$-Laplacian can be found in
Alves-Assun\c{c}ao-Miya\-ga\-ki~\cite{Alves},
Axelsson-Keith-McIntosh~\cite{Axelsson-Keith-McIntosh},
Bai-Papageorgiou-Zeng~\cite{Bai-Papageorgiou-Zeng2021MZ},
Barboteu-Bartosz-Han-Janiczko~\cite{BBHT2015SIAM}, Zeng-Bai-Gasi\'nski-Winkert~\cite{Zeng-Bai-Gasinski-Winkert-2020},
Duvaut-Lions~\cite{DL},
Figueiredo~\cite{Fig},
Gasi\'nski-Papageorgiou~\cite{Gasinski-Papageorgiou2017CVPDE},
Gasi\'nski-Winkert~\cite{Gasinski-Winkert2020JDE},
Han~\cite{WeiminHanSIAM2020},
Liu-Motreanu-Zeng~\cite{Liu-Motreanu-Zeng2019CVPDE,Liu-Motreanu-Zeng2021SIOP},
Maz'ya-Rossmann~\cite{Mazya}, Zeng-R\v adulescu-Winkert~\cite{Zeng-Radulescu-Winkert-SIMA2022},
Mig\'orski-Khan-Zeng~\cite{MKZ2019IV,MKZ2020IV},
Mihailes\-cu-R\v adulescu~\cite{Mihailescu-Radulescu2011IJM},
Mitrea~\cite{Mitrea},
Papageorgiou-Qin-R\v adulescu~\cite{Papageorgiou-Qin-Radulescu2021BSM}, Liu-Papageorgiou~\cite{Liu-Papageorgiou-ANONA},
Papageorgiou-R\v adules\-cu-Repov\v s~\cite{Papageorgiou-Radulescu-Repovs2021Nonlinearity,
Papageorgiou-Radulescu-Repovs2020JMPA} and  Yu-Feng~\cite{Yu-Feng-ANONA}.
Results on convergence of optimal solutions in
optimal control problems can be found in
Denkowski-Mig\'orski~\cite{DENMIG87}, Gariboldi-Tarzia~\cite{Gariboldi-Tarzia-2003-AMO},
Denkowski-Mortola~\cite{DENMOR}, Zeng-Mig\'orski-Liu~\cite{Zeng-Migorski-Liu2021SIOPT}, Mig\'orski~\cite{Migorski95,Migorski_ANN95}, Denkowski-Mig\'orski-Papageorgiou~\cite[Section~4.2]{DMP2}, Liu-Mig\'orski-Nguyen-Zeng~\cite{Liu-Migorski-Nguyen-Zeng-AMSC2021}, Zeng-Mig\'orski-Khan~\cite{Zeng-Migorski-Khan-2021}
and the references therein.

The purpose of this paper is fourfold.
The first goal is to prove the unique weak solvability of Problems~\ref{problemss1} and~\ref{problemss2}
by applying a surjectivity theorem for pseudomonotone operators.
The second purpose is to establish a comparison principle and a monotonicity result for solutions of Problems~\ref{problemss1} and~\ref{problemss2}.
The third aim is to deliver a convergence result which shows that the solution of Problem~\ref{problemss1} can be approached by
the solution of Problem~\ref{problemss2},
as $\alpha\to \infty$.
Moreover, our last intention is to investigate two optimal control problems, Problems~\ref{problemss3} and~\ref{problemss4}, and to examine the asymptotic behavior of optimal solutions
(i.e., control-state pairs) and of minimal values
for Problem~\ref{problemss4}, when parameter $\alpha$
in the boundary condition, representing for instance a heat transfer coefficient, tends to infinity.

The rest of the paper is organized as follows.
In Section~\ref{Section2}, we recall basic notation and collect the necessary preliminary material.
Section~\ref{Section3} is devoted to the proof of  existence and uniqueness of solutions to Problems~\ref{problemss1} and~\ref{problemss2}, and to discuss a comparison principle as well as a convergence result to Problems\ref{problemss1} and~\ref{problemss2}. Finally, in Section~\ref{Section4}, we introduce two optimal control problems governed by Problems~\ref{problemss1} and Problem~\ref{problemss2}, respectively, and explore the asymptotic behavior of the optimal controls and system states
to Problem~\ref{problemss4}.

\section{Mathematical background}\label{Section2}

In this section, we review some basic notation, definitions and the necessary preliminary material, which will be used in next sections.
More details can be found, for instance, in~\cite{DMP2,smo1,Papageorgiou-Radulescu-Repovs2019,SMHistory}.

Let $(Y,\|\cdot\|_Y)$ be a Banach space and
$Y^*$ stand for the dual space to $Y$.
We denote by $\langle\cdot,\cdot\rangle$ the duality brackets for the pair of $Y^*$ and $Y$.
Everywhere below, the symbols $\wto$ and $\to$
represent the weak and strong convergences,
respectively.
We say that a mapping $F\colon Y\to Y^{*}$
is
\begin{itemize}
\item[(i)] monotone, if
\begin{equation*}
\langle F u - F v, u - v\rangle\ge0
\ \ \mbox{for all}\ \ u, v\in Y,
\end{equation*}
\item[(ii)] strictly monotone, if
\begin{equation*}
\langle F u - F v, u-v\rangle>0
\ \ \mbox{for all}\ \ u, v\in Y
\ \mbox{with} \ u\neq v,
\end{equation*}
\item[(iii)] of type $(S)_+$ (or $F$ satisfies  $(S_+)$--property),
if for any sequence $\{u_n\} \subset Y$ with
$u_n\wto u$ in $Y$ as $n\to\infty$ for some $u \in Y$
and $\displaystyle
\limsup_{n\to\infty} \langle Fu_n, u_n-u\rangle\le 0$, then the sequence $\{u_n\}$ converges strongly to $u$ in $Y$,
\item[(iv)] pseudomonotone, if it is bounded and for every sequence $\{ u_n \} \subseteq Y$ converging weakly to $u \in Y$ with
$\displaystyle
\limsup_{n\to\infty} \langle F u_n, u_n - u \rangle \le 0$, then
$$
\displaystyle
\langle F u, u - v \rangle \le \liminf_{n\to\infty}  \langle F u_n, u_n - v \rangle
\ \ \mbox{for all}\ \ v \in Y,
$$
\item[(v)]
coercive, if
$$
\lim_{\| v \|_Y \to \infty}
\frac{\langle Fv, v \rangle}{\|v\|_Y} = +\infty.
$$
\end{itemize}

It is not difficult to see that if $F$ is of type $(S)_+$, then $F$ is pseudomonotone as well. Note that the operator $F \colon Y \to Y^*$ is pseudomonotone
if and only if it is bounded and $y_n \to y$ weakly in $Y$ with
$\displaystyle
\limsup_{n\to\infty}  \langle F y_n, y_n - y \rangle \le 0$
entails
$\displaystyle
\lim_{n\to\infty}  \langle F y_n, y_n - y \rangle = 0$ and $F y_n \to F y$ weakly in $Y^*$.
Furthermore, if $F \in {\mathcal L}(Y, Y^*)$
(the class of linear and bounded operators)
is nonnegative, then it is pseudomonotone.

\begin{Theorem}\label{pseudomonotheorem}
Let $Y$ be a Banach space, and $F$, $G\colon Y\to Y^*$. Then, we have

\smallskip

{\rm(i)} if $F$
is bounded, hemicontinuous, and monotone,
then $F$ is pseudomonotone,

\smallskip

{\rm(ii)} if $F$ and $G$ are pseudomonotone,
then $F+G$ is also pseudomonotone.
\end{Theorem}

The class of pseudomonotone and coercive operators
enjoys the well-known surjectivity property.
\begin{Theorem}\label{surjecttheorem}
Let $Y$ be a Banach space and $F\colon Y\to Y^*$ be pseudomonotone and
coercive. Then $F$ is surjective, i.e., for any $f\in Y^*$, there is at least one solution to
the equation $Fu= f$.
\end{Theorem}

Let $\Omega\subset \real^N$  be a bounded domain such that its Lipschitz boundary $\Gamma=\partial \Omega$ is divided into three measurable and mutually disjoint parts $\Gamma_1$, $\Gamma_2$, and $\Gamma_3$, and $\Gamma_1$ has a positive measure.
Let $1< p<+\infty$ and
$p'>1$ be the conjugate exponent of $p$, i.e., $\frac{1}{p}+\frac{1}{p'}=1$.
In the sequel, we denote by $p^*$ the critical exponent to $p$ given by
\begin{equation}\label{defp1*}
p^*=
\begin{cases}
\frac{Np}{N-p} &\text{ if }\ p<N,\\
+\infty &\text{ if }\ p\ge N.
\end{cases}
\end{equation}
Throughout the paper, the norms of the Lebesgue space $L^p(\Omega)$ and Sobolev space $W^{1,p}(\Omega)$ are defined by
\begin{equation*}
\|u\|_{L^p(\Omega)}:=\left(\int_\Omega |u(x)|^p\,dx\right)^\frac{1}{p}
\ \ \mbox{for all}\ \ u\in L^p(\Omega),
\end{equation*}
and
\begin{equation*}
\|u\|_{W^{1,p}(\Omega)}:=\Big(\|u\|_{L^p(\Omega)}^p
+\|\nabla u\|_{L^p(\Omega;\real^N)}^p\Big)^{1/p}
\ \ \mbox{for all} \ \ u\in W^{1,p}(\Omega),
\end{equation*}
respectively.
We introduce a subspace $V$ of $W^{1,p}(\Omega)$ given by
\begin{equation*}
V:=\{u\in W^{1,p}(\Omega)\,\mid\,u=0\mbox{ on $\Gamma_1$}\}.
\end{equation*}
From the fact
that $\Gamma_1$ has a positive measure and by
the Poincar\'e inequality, it follows that
$V$ endowed with the norm
\begin{equation*}
\|u\|_V:=\left(\int_\Omega|\nabla u|^p\,dx\right)^\frac{1}{p}
\ \ \mbox{for all} \ \ u\in V
\end{equation*}
is a reflexive Banach space.
Further, we consider the subsets $K$ and $K_0$ of $V$ defined by
\begin{eqnarray}
&&
K:=\{ u\in V \mid u=b
\ \ \mbox{on} \ \ \Gamma_3 \},
\label{KKK}
\\[1mm]
&&
K_0:=\{ u\in V \mid u=0
\ \ \mbox{on} \ \ \Gamma_3 \}, \label{KK0}
\end{eqnarray}
respectively, where $b > 0$ is given in Problem~\ref{problemss1}.

We end the section with the nonlinear operator
$A \colon V \to V^*$ defined by
\begin{equation}\label{operator_representation}
\langle Au, v\rangle=
\int_\Omega\left(|\nabla u(x)|^{p-2}\nabla u(x)
+\mu\, |\nabla u(x)|^{q-2}\nabla u(x),\nabla v(x)\right)_{\real^N}\,dx
\end{equation}
for all $u$, $v \in V$.
The following result summarizes the main properties of this map (see, e.g.,~\cite[Chapter~3, Example~1.7, p. 303]{HP}).
\begin{Proposition}\label{proA}
Let $\mu>0$ and $1<q<p<+\infty$.
Then, the operator $A \colon V\to V^*$ defined by $(\ref{operator_representation})$
is bounded, continuous, strictly monotone (hence maximal monotone) and of type $(S_+)$.
\end{Proposition}

\section{Existence, uniqueness and con\-ver\-gen\-ce results}\label{Section3}

This section is devoted to study the unique solvability
of Problems~\ref{problemss1} and~\ref{problemss2}.
We discuss a comparison principle which reveals
the essential relations between the unique weak solutions of Problems~\ref{problemss1} and~\ref{problemss2} as well as the constant $b>0$.
We also establish a monotonicity property
of solution to Problem~\ref{problemss2} with respect to the parameter $\alpha$, and
obtain a convergence result which shows that the unique solution to Problem~\ref{problemss1} can be approached by the unique solution to Problem~\ref{problemss2}
when the parameter $\alpha $
tends to infinity.

Let us consider the nonlinear operators
$B$, $L\colon V\to V^*$ defined
by
\begin{equation}\label{eqnss3.2}
\langle Bu,v\rangle
:=\int_\Omega\beta|u(x)|^{\theta-2}u(x)v(x)\,dx
\ \ \mbox{for all} \ \ u, v\in V,
\end{equation}
and
\begin{equation}\label{eqnss3.021}
\langle Lu,v\rangle
:=\int_{\Gamma_3}l(x,u(x)) v(x)\,d\Gamma
\ \ \mbox{for all}\ \ u, v\in V,
\end{equation}
respectively. By the Riesz representation theorem, we introduce the function  $f\in V^*$ defined by
\begin{equation}\label{eqnss3.11}
\langle f,v\rangle=\int_\Omega g(x)v(x)\,dx-\int_{\Gamma_2}r(x)v(x)\,d\Gamma
\ \ \mbox{for all}\ \ v\in V.
\end{equation}
Using the notation above, it is not difficult to see that Definition~\ref{DEF1} can be equivalently
rewritten as follows:
\lista{\it
	\item[$\rm (i)'$] a function
	$u_\alpha\in V$ is called to be a weak solution of Problem~$\ref{problemss2}$ associated with $\alpha>0$, if it satisfies
	\begin{equation}\label{eqnss3.12}
	\langle Au_\alpha+Bu_\alpha,v\rangle
	+\alpha\, \langle Lu_\alpha,v\rangle
	= \langle f,v\rangle
	\ \ \mbox{\rm for all} \ \ v \in V,
	\end{equation}
	\item[$\rm (ii)'$] a function
	$u_\infty\in K$ is called to be a weak solution of Problem~$\ref{problemss1}$, if
	\begin{equation}\label{eqnss3.13}
	\langle A u_\infty + B u_\infty, v\rangle
	= \langle f, v\rangle
	\ \ \mbox{\rm for all} \ \ v \in K_0,
	\end{equation}
	where
	$K$ and $K_0$ are given by $(\ref{KKK})$ and $(\ref{KK0})$, respectively.
}

\medskip

We assume that the function $l$ in the operator $L$ satisfies the following hypotheses.

\smallskip

\noindent
${\underline{H(l)}}$: \ \ $l\colon \Gamma_3\times \real\to \real$ is a Carath\'eodory function
(i.e., for all $s\in \real$, the function $x\mapsto l(x,s)$ is measurable,
and for a.e. $x\in \Gamma_3$,
$s\mapsto l(x,s)$ is continuous) such that

\smallskip

\lista{
\item[(i)] there exist $a_l\in L^{p'}_+(\Gamma_3)$ and $b_l>0$ satisfying
\begin{equation*}
|l(x,s)|\le a_l(x)+b_l\, (1+|s|^{p-1})
\end{equation*}
for all $s\in\real$ and a.e. $x\in \Gamma_3$,
\smallskip
\item[(ii)] for a.e. $x\in \Gamma_3$, $s\mapsto l(x,s)$ is nondecreasing, i.e., it satisfies
\begin{equation*}
(l(x,s_1)-l(x,s_2))(s_1-s_2)\ge 0
\end{equation*}
for all $s_1,s_2\in \real$ and a.e. $x\in \Gamma_3$,
\smallskip
\item[(iii)] for a.e. $x\in \Gamma_3$, $l(x,s)=0$ if and only if $s=b$.
}

We next give a concrete example for function $l$ which satisfies hypotheses $H(l)$.
\begin{Example}
Let $b>0$ be given in Problem~$\ref{problemss1}$
and ${\rm sgn} \colon \real\to \{-1,0,1\}$ be the sign function, namely,
\begin{eqnarray*}
{\rm sgn}(s):=\left\{\begin{array}{lll}
1&\mbox{\rm if $s>0$},\\
0&\mbox{\rm if $s=0$},\\
-1&\mbox{\rm if $s<0$}.
\end{array}\right.
\end{eqnarray*}
Also, let $1<p<+\infty$ and
$\omega\in L^\infty(\Gamma_3)$, $\omega > 0$. Then, the function $l\colon\Gamma_3\times \real\to \real$  defined by
\begin{equation*}
l(x,s)=\omega(x)\, {\rm sgn}(s-b)\, |s-b|^{p-1}
\ \ \mbox{\rm for all} \ \ s\in\real\ \ \mbox{\rm and}
\ \ x\in\Gamma_3,
\end{equation*}
satisfies hypotheses $H(l)$.
\end{Example}

Besides, we need the following assumption.

\smallskip

\noindent
${\underline{H(0)}}$: \ \
$g\in L^{p'}(\Omega)$ with $g\le 0$ in $\Omega$,
$r\in L^{p'}(\Gamma_2)$ with $r\ge 0$ on $\Gamma_2$,
and $b>0$.

\smallskip

The main results on existence, uniqueness, comparison, monotonicity and convergence to Problem~\ref{problemss2}
are provided in the following theorem.
\begin{Theorem}\label{theorem1}
Assume that $H(l)$ and $H(0)$ are fulfilled.
Then, we have
\begin{itemize}
\item[{\rm (i)}]
Problem~$\ref{problemss1}$ has a unique solution $u_\infty\in K$,
\item[{\rm (ii)}]
for every $\alpha>0$, Problem~$\ref{problemss2}$
has a unique solution $u_\alpha\in V$,
\item[{\rm (iii)}]
$u_\infty\le b$ in $\Omega$,
\item[{\rm (iv)}]
for every $\alpha>0$, it holds $u_\alpha\le b$ in $\Omega$ and $u_\alpha\le b$ on $\Gamma_3$,
\item[{\rm (v)}]
for every $\alpha>0$, it holds $u_\alpha\le u_\infty$ in $\Omega$,
\item[{\rm (vi)}]
if $0<\alpha_1\le \alpha_2$, then
$u_{\alpha_1}\le u_{\alpha_2}$ in $\Omega$,
\item[{\rm (vii)}]
if a sequence $\{\alpha_n\}$ is such that $\alpha_n>0$ for all $n\in\mathbb N$ with $\alpha_n\to \infty$ as $n\to\infty$, then $u_{\alpha_n}\to u_{\infty}$
in $V$ as $n\to\infty$.
\end{itemize}
\end{Theorem}

\begin{proof}
(i) It is a direct consequence
of~\cite[Lemma~6]{ZengMigorskiTarzia2021AA}.

\smallskip

(ii) From (\ref{eqnss3.12}), we can observe that $u\in V$ is a weak solution to Problem~\ref{problemss2}
if and only if it solves the following abstract operator equation: find $u\in V$ such that
\begin{equation}\label{incproblem}
Au+Bu+\alpha Lu= f \ \ \mbox{in} \ \ V^*.
\end{equation}
By Proposition~\ref{proA}, we know that $A$ is a
bounded, continuous, strictly monotone
(hence maximal monotone) operator,
and of type $(S_+)$. Also, we can obtain
\begin{equation}\label{eqnss3.3}
\|Au\|_{V^*}\le \|u\|_V^{p-1} +
\mu\, \|\nabla u\|_{L^{(q-1)p'}(\Omega;\real^N)}^{q-1}
\ \ \mbox{for all} \ \ u\in V.
\end{equation}
Employing~\cite[Theorem~3.69]{smo1} we deduce
that $A$ is a pseudomonotone operator.
As concerns operator $B$, it is monotone and continuous, and satisfies
\begin{equation}\label{eqnss3.4}
\|Bu\|_{V^*}\le c_1\, \|u\|_V^{\theta-1}
\ \ \mbox{for all} \ \ u\in V
\end{equation}
with some $c_1>0$.
The latter combined with the compactness of the embedding of $V$ to $L^\theta(\Omega)$
(due to $\theta<p^*$)
implies that $B$ is completely continuous, so,
it is also pseudomonotone. For any $u\in V$, from hypotheses $H(l)$ and the H\"older inequality,
we have
\begin{eqnarray*}
&&\hspace{-0.4cm}
\|Lu\|_{L^{p'}(\Omega)}=\sup_{v\in L^p(\Gamma_3),\|v\|_{L^p(\Gamma_3)}=1}\langle Lu,v\rangle_{L^{p'}(\Gamma_3)\times L^p(\Gamma_3)}
\\
&&
\le \sup_{v\in L^p(\Gamma_3),\|v\|_{L^p(\Gamma_3)}=1}
\int_{\Gamma_3}|l(x,u(x))v(x)|\,d\Gamma
\\
&&
\le \sup_{v\in L^p(\Gamma_3),\|v\|_{L^p(\Gamma_3)}=1}\int_{\Gamma_3}(a_l(x)+b_l(1+|u(x)|^{p-1}))|v(x)|\,d\Gamma
\\
&&
\le \sup_{v\in L^p(\Gamma_3),\|v\|_{L^p(\Gamma_3)}=1}
\left(\|a_l\|_{L^{p'}(\Gamma_3)}
+ b_l|\Gamma_3|^\frac{1}{p'} + b_1\|u\|_{L^p(\Gamma_3)}^{p-1}
\right)\, \|v\|_{L^p(\Gamma_3)}
\\
&& \le \|a_l\|_{L^{p'}(\Gamma_3)}
+ b_l \, |\Gamma_3|^\frac{1}{p'} + b_1 \, \|u\|_{L^p(\Gamma_3)}^{p-1}.
\end{eqnarray*}
Hence,
$L\colon V\to V^*$ is well-defined.
From the compactness of the trace operator
$\gamma\colon V\to L^p(\Gamma_3)$ and the definition of $L$, we can also see that $L$ is continuous.
Besides, we use condition $H(l)$(ii) to infer that
$L$ is monotone, that is,
\begin{equation*}
\langle Lu-Lv,u-v\rangle
=\int_{\Gamma_3}(l(x,u(x))-l(x,v(x)))(u(x)-v(x))
\,d\Gamma\ge 0
\end{equation*}
for all $u$, $v\in V$.
This together with~\cite[Theorem~3.69]{smo1} implies  that $L$ is a pseudomonotone operator. Therefore,
by using Theorem~\ref{pseudomonotheorem}(ii),
we infer that $A+B+\alpha L\colon V\to V^*$ is pseudomonotone.

Next, let $\varepsilon>0$ be arbitrary.
From hypothesis $H(l)$, we get the estimate
\begin{eqnarray*}
&&\hspace{-0.8cm}
\langle Lu,u\rangle
= \int_{\Gamma_3}l(x,u(x))u(x)\,d\Gamma
=\int_{\Gamma_3}l(x,u(x))(u(x)-b)\,d\Gamma
+\int_{\Gamma_3}l(x,u(x))b\,d\Gamma
\\
&&\ge \int_{\Gamma_3}l(x,b)(u(x)-b)\,d\Gamma
+\int_{\Gamma_3}l(x,u(x))b\,d\Gamma
\\
&&\ge -\int_{\Gamma_3}(a_l(x)+b_l(1+|u(x)|^{p-1}))b\,d\Gamma
\\
&& \ge -b\, |\Gamma_3|^\frac{1}{p}\|a_l\|_{L^{p'}(\Gamma_3)}
-b_l\, b\, |\Gamma_3|- b_l\, b\, \int_{\Gamma_3}|u(x)|^{p-1}\,d\Gamma
\\
&&\ge
-b\,|\Gamma_3|^\frac{1}{p}\|a_l\|_{L^{p'}(\Gamma_3)} -b_l\, b\, |\Gamma_3| -\varepsilon\, \|u\|_{L^{p}(\Gamma_3)}^p - c(\varepsilon)
\end{eqnarray*}
with some $c(\varepsilon)>0$,
where the last inequality is obtained by using
the Young inequality. From the estimates above
and definitions of $A$ and $B$, we obtain
\begin{eqnarray}\label{eqns3.9}
&&\hspace{-0.8cm}
\langle Au+Bu+\alpha Lu, u\rangle
\\
&&\ge \|u\|_{V}^p
+\mu\|\nabla u\|_{L^q(\Omega;\real^N)}^q
+\beta\|u\|_{L^\theta(\Omega)}^\theta
-\alpha \, b\, |\Gamma_3|^\frac{1}{p}
\|a_l\|_{L^{p'}(\Gamma_3)}
-\alpha \, b_l\, b\, |\Gamma_3| \nonumber
\\
&&
\qquad -\varepsilon\,\alpha\,\|u\|_{L^{p}(\Gamma_3)}^p
-\alpha \, c(\varepsilon)\nonumber
\\
&&\ge \|u\|_{V}^p
+\mu\|\nabla u\|_{L^q(\Omega;\real^N)}^q
+\beta\|u\|_{L^\theta(\Omega)}^\theta
-\alpha\, b\, |\Gamma_3|^\frac{1}{p}
\|a_l\|_{L^{p'}(\Gamma_3)}
-\alpha \, b_l\, b\, |\Gamma_3| \nonumber
\\
&&
\qquad -\varepsilon\, \alpha \, c_V^p \, \|u\|_V^p
-\alpha \, c(\varepsilon)\nonumber
\end{eqnarray}
for all $u\in V$.
We set $\varepsilon=\frac{1}{2 \alpha c_V^p}$ ,
where $c_V > 0$ is the constant for the embedding
of $V$ into $L^p(\Gamma_3)$, that is,
$\| v \|_{L^p(\Gamma_3)} \le c_V \| v \|_V$ for $v \in V$.
Then, because of $p>1$,
we conclude that $A+B+\alpha L$ is coercive.
Therefore, all conditions of Theorem~\ref{surjecttheorem} are verified.
Using this theorem, we deduce that Problem~$\ref{problemss2}$ has at least one solution.  Furthermore,
the strict monotonicity of $A$ allows us to apply
a standard method to show that Problem~$\ref{problemss2}$
has a unique solution $u_\alpha \in V$.

\smallskip

(iii) Let $u_\infty\in K$ be the unique solution of Problem~\ref{problemss1}. We set $w=(u_\infty-b)^+$. Then, one has $u_\infty=b$ on $\Gamma_3$ and $w=0$ on $\Gamma_3$. Thus, $w\in K_0$.
We take $v= w$ in (\ref{eqnss3.13}) to get
\begin{eqnarray*}
&&\hspace{-0.9cm}
\int_\Omega\left(|\nabla u_\infty|^{p-2}\nabla u_\infty+\mu|\nabla u_\infty|^{q-2}\nabla u_\infty,\nabla (u_\infty-b)^+\right)_{\real^N}\,dx
\\
&&+\int_\Omega\beta
|u_\infty(x)|^{\theta-2}u_\infty(x)(u_\infty(x)-b)^+\,dx
=\langle A u_\infty + B u_\infty, w\rangle
= \langle f, w\rangle
\\
&&
=\int_\Omega g(x)(u_\infty(x)-b)^+\,dx
-\int_{\Gamma_2}r(x)(u_\infty(x)-b)^+\,d\Gamma.
\end{eqnarray*}
From condition $H(0)$ we deduce
\begin{equation*}
\int_\Omega g(x)(u_\infty(x)-b)^+\,dx
-\int_{\Gamma_2}r(x)(u_\infty(x)-b)^+\,d\Gamma\le 0,
\end{equation*}
while the monotonicity of $B$ and
nonnegativity of $b$ guarantee that
\begin{equation*}
\int_\Omega\beta|u_\infty(x)|^{\theta-2}u_\infty(x)(u_\infty(x)-b)^+\,dx \ge 0.
\end{equation*}
Taking into account the last two  inequalities
and the fact $\nabla b=0$, we have
\begin{equation*}
\langle A u_\infty-Ab,w\rangle \le 0.
\end{equation*}
The latter combined with the strict monotonicity of $A$ implies that $w=0$. This means that $u_\infty\le b$ in $\Omega$.

\smallskip

(iv) Let $u_\alpha\in V$ be the unique solution of Problem~\ref{problemss2} corresponding to $\alpha>0$.
We put $w=(u_\alpha-b)^+$. Inserting $v = w$ into (\ref{eqnss3.12}), it yields
\begin{equation}\label{intee}
\langle Au_\alpha+Bu_\alpha,w\rangle
+\alpha\langle Lu_\alpha,w\rangle
= \langle f, w \rangle.
\end{equation}
Hence, we have
\begin{eqnarray*}
&&\hspace{-0.6cm}
\langle Au_\alpha,w\rangle\le -\alpha\langle Lu_\alpha,w\rangle=
-\alpha\int_{\Gamma_3}
l(x,u_\alpha(x))(u_\alpha(x)-b)^+\,d\Gamma
\\
&&=
-\alpha\int_{\{u_\alpha>b\}\cap\Gamma_3}
l(x,u_\alpha(x))(u_\alpha(x)-b)\,d\Gamma
\le -\alpha\int_{\{u_\alpha>b\}\cap\Gamma_3}
l(x,b)(u_\alpha(x)-b)\,d\Gamma\\
&& =0,
\end{eqnarray*}
where we have used the monotonicity of
the function $s\mapsto l(x,s)$ and hypothesis $H(l)$(iii), and the set $\{u_\alpha>b\}$ is defined by
$\{u_\alpha>b\}:=\{x\in \Gamma_3 \mid u_\alpha(x)>b\}$.
Therefore, one has
\begin{equation*}
\langle Au_\alpha-Ab,w\rangle\le 0.
\end{equation*}
Then, it is true that $w=0$, i.e.,
$u_\alpha \le b$ in $\Omega$.

From equation (\ref{intee}) and the fact
$u_\alpha\le b$ in $\Omega$, we have
\begin{eqnarray}\label{ineess}
&&
0\ge \alpha\, \langle Lu_\alpha,w\rangle=\alpha \int_{\Gamma_3}l(x,u_\alpha(x))
(u_\alpha(x)-b)^+\,d\Gamma
\\
&& \qquad
= \, \alpha \int_{\{u_\alpha>b\}\cap\Gamma_3}
l(x,u_\alpha(x))(u_\alpha(x)-b)\,d\Gamma.
\nonumber
\end{eqnarray}
Using the monotonicity of $s\mapsto l(x,s)$ and hypothesis $H(l)$(iii) again, we obtain
$l(x,u_\alpha(x))$ $(u_\alpha(x)-b)\ge
l(x,b)(u_\alpha(x)-b)=0$
for a.e. $x\in \{u_\alpha>b\} \cap \Gamma_3$.
This together with inequality (\ref{ineess})
implies $l(x,u_\alpha(x))=0$ due to $u_\alpha(x)>b$.
On the other hand, condition $H(l)$(iii) turns out  $u_\alpha(x)=b$. This leads to a contradiction.
Therefore, we conclude that $u_\alpha\le b$ on $\Gamma_3$ as claimed.

\smallskip

(v) For any $\alpha>0$ fixed, let $u_\alpha\in V$ and $u_\infty\in K$ be the unique solutions of Problem~\ref{problemss2} and Problem~\ref{problemss1}, respectively.
We set $w=(u_\alpha-u_\infty)^+$.
From assertion (iv) it follows that
$u_\alpha\le b$ on $\Gamma_3$.
We use the definition of $K$
(i.e., $u_\infty=b$ on $\Gamma_3$) to get $w=(u_\alpha-u_\infty)^+=0$ on $\Gamma_3$,
so, $w\in K_0$.
Taking $v=w$ into (\ref{eqnss3.12}) and (\ref{eqnss3.13}), respectively, we have
\begin{equation*}
\langle A u_\infty + B u_\infty, w\rangle
= \langle f, w\rangle
\ \ \ \mbox{and} \ \ \
\langle Au_\alpha+Bu_\alpha,w\rangle
+\alpha\langle \,Lu_\alpha,w\rangle= \langle f, w\rangle.
\end{equation*}
Summing up the equalities above, it gives
\begin{eqnarray*}
&&\hspace{-0.6cm}
\langle Au_\alpha- A u_\infty+Bu_\alpha-B u_\infty,w\rangle
=-\alpha\, \langle Lu_\alpha,w\rangle
\\
&&=-\alpha \int_{\Gamma_3}l(x,u_\alpha(x))(u_\alpha(x)-u_\infty(x))^+\,d\Gamma
= -\alpha \int_{\Gamma_3}l(x,u_\alpha(x))(u_\alpha(x)-b)^+\,d\Gamma\\
&&=0.
\end{eqnarray*}
Employing the monotonicity of $A$ and $B$, we have $w=0$, which implies $u_\alpha\le u_\infty$ in $\Omega$.

\smallskip

(vi) Let $0<\alpha_1\le \alpha_2$ and $u_i:=u_{\alpha_i}$ be the unique solution of Problem~\ref{problemss2} associated with $\alpha=\alpha_i$ for $i=1$, $2$.
Hence
\begin{equation*}
\langle Au_i+Bu_i,v\rangle
+\alpha_i\langle Lu_i,v\rangle= \langle f,v\rangle
\ \ \mbox{for all} \ \ v \in V.
\end{equation*}
Let $w:=(u_1-u_2)^+$.
Putting $v=w$ into the equality above and summing up the resulting equations, one has
\begin{eqnarray}\label{eeqss1}
&&\hspace{-0.5cm}
\langle Au_1-Au_2+Bu_1-Bu_2,w\rangle
=\langle \alpha_2Lu_2-\alpha_1Lu_1,w\rangle=
\alpha_2\, \big\langle Lu_2-\frac{\alpha_1}{\alpha_2}Lu_1,w\big\rangle
\nonumber\\
&&
=\alpha_2\int_{\Gamma_3}\left(l(x,u_2(x))-\frac{\alpha_1}{\alpha_2}l(x,u_1(x))\right)(u_1(x)-u_2(x))^+\,d\Gamma\nonumber\\
&&
=\alpha_2\int_{\{u_1>u_2\}\cap\Gamma_3}\left(l(x,u_2(x))-\frac{\alpha_1}{\alpha_2}l(x,u_1(x))\right)(u_1(x)-u_2(x))\,d\Gamma.
\end{eqnarray}
Recalling that $u_i\le b$ on $\Gamma_3$
(see assertion (iv)),
we use condition $H(l)$(ii) to find that
$0=l(x,b)\ge l(x,u_i(x))$ for a.e. $x\in \Gamma_3$ and $i=1,2$.
From hypothesis $H(l)$(ii) again and the fact $\frac{\alpha_1}{\alpha_2}\le 1$, we get
\begin{eqnarray*}
&&\hspace{-0.9cm}
-\frac{\alpha_1}{\alpha_2}\, l(x,u_1(x))(u_1(x)-u_2(x))
\le -\frac{\alpha_1}{\alpha_2} \,
l(x,u_2(x))(u_1(x)-u_2(x))
\\[1mm]
&&
\le - l(x,u_2(x))(u_1(x)-u_2(x))\mbox{ for a.e. $x\in \{u_1>u_2\}$}.
\end{eqnarray*}
Inserting the inequality above into (\ref{eeqss1}) and using hypothesis $H(l)$(ii), it yields
\begin{eqnarray*}
&&\langle Au_1-Au_2+Bu_1-Bu_2,w\rangle\\
&&\le \alpha_2\int_{\{u_1>u_2\}\cap \Gamma_3}\left(l(x,u_2(x))- l(x,u_1(x))\right)(u_1(x)-u_2(x))\,d\Gamma \le 0.
\end{eqnarray*}
Therefore, we conclude that $w=0$, and finally
$u_1\le u_2$ in $\Omega$.

\smallskip

(vii) Let $\{\alpha_n\}$ be a sequence such that $\alpha_n>0$ for all $n\in\mathbb N$ and $\alpha_n\to \infty$ as $n\to\infty$.
For each $n\in\mathbb N$, let
$u_n:=u_{\alpha_n}$ be the unique solution of Problem~\ref{problemss2} corresponding to
$\alpha= \alpha_n$.
We claim that sequence $\{u_n\}$ is bounded in $V$.
For each $n\in \mathbb N$, we have
\begin{equation*}
\|f\|_{V^*}(\|u_n\|_V+\|u_\infty\|_V)\ge \langle f,u_n-u_\infty\rangle=\langle Au_n+Bu_n+\alpha_nLu_n,u_n-u_\infty\rangle.
\end{equation*}
Applying conditions $H(l)$(ii) and (iii), one finds
\begin{eqnarray*}
&&\hspace{-0.9cm}
\alpha_n\langle Lu_n,u_n-u_\infty\rangle
=\alpha_n\int_{\Gamma_3}l(x,u_n(x))(u_n(x)-u_\infty(x))\,d\Gamma
\\
&&
=\alpha_n\int_{\Gamma_3}l(x,u_n(x))(u_n(x)-b)\,d\Gamma
\ge \alpha_n\int_{\Gamma_3}l(x,b)(u_n(x)-b)\,d\Gamma =0.
\end{eqnarray*}
Then, we have
\begin{equation*}
\|f\|_{V^*}(\|u_n\|_V+\|u_\infty\|_V)+\langle Au_n+Bu_n,u_\infty\rangle\ge \langle Au_n+Bu_n,u_n\rangle.
\end{equation*}
We use the H\"older inequality and the monotonicity of $B$ to get
\begin{eqnarray*}
&&\hspace{-0.6cm}
\|u_n\|_V^p+\mu\|\nabla u_n\|_{L^q(\Omega;\real^N)}^q
\le \|u_n\|_V^{p-1}\|u_\infty\|_V+\beta\|u_\infty\|_{L^\theta(\Omega)}^\theta
\\[1mm]
&&
+\, \mu\,
\|\nabla u_n\|_{L^{p'(q-1)}(\Omega;\real^N)}^{q-1}
\|u_\infty\|_V+ \|f\|_{V^*}\left(\|u_n\|_V+\|u_\infty\|_V\right)
+c_2\|u_\infty\|_{L^{(\theta-1)p'}(\Omega)}^{\theta-1}\|u_n\|_V
\end{eqnarray*}
for some $c_2>0$. This reveals that sequence $\{u_n\}$ is bounded in $V$. Passing to the sub\-sequence if necessary, we may assume that
$u_n\wto u$ in $V$ as $n\to\infty$
with some $u\in V$.
We are going to show that $u\in K$, i.e., $u=b$ on $\Gamma_3$. The boundedness of operators $A$, $B$
and of sequence $\{u_n\}$ guarantee that there
exists a constant $c_3>0$ independent of $n$
such that
\begin{equation*}
\langle Au_n+Bu_n-f,u_\infty-u_n\rangle\le c_3
\end{equation*}
for all $n\in\mathbb N$. Then, for each $n\in\mathbb N$, we have
\begin{equation*}
\langle Lu_n,u_n-u_\infty\rangle
= \frac{1}{\alpha_n}\langle Au_n+Bu_n-f,u_\infty-u_n\rangle
\le \frac{c_3}{\alpha_n}.
\end{equation*}
Keeping in mind that the embedding of $V$ into $L^p(\Gamma_3)$ is compact, we have $u_n\to u$ in $L^p(\Gamma_3)$. By the Lebesgue dominated convergence theorem, we get
\begin{eqnarray*}
&&\hspace{-0.6cm}
0 =\lim_{n\to\infty} \frac{c_3}{\alpha_n}\ge \lim_{n\to\infty}\langle Lu_n,u_n-u_\infty\rangle=\lim_{n\to\infty}\int_{\Gamma_3}l(x,u_n(x))(u_n(x)-u_\infty(x))\,d\Gamma
\\
&&
=\int_{\Gamma_3}l(x,u(x))(u(x)-u_\infty(x))\,d\Gamma=\int_{\Gamma_3}l(x,u(x))(u(x)-b)\,d\Gamma
\\
&& \ge \int_{\Gamma_3}l(x,b)(u(x)-b)\,d\Gamma=0.
\end{eqnarray*}
So, it holds $l(x,u(x))(u(x)-b)=0$ for a.e.
$x\in \Gamma_3$. Condition $H(l)$(iii) points out
that $u(x)=b$ for a.e. $x\in \Omega$.
This means that $u\in K$.

Subsequently, we shall show that $u=u_\infty$. For any $w\in K$, we have
\begin{equation*}
\langle Au_n+Bu_n,u_n-w\rangle+\alpha_n\langle Lu_n,u_n-w\rangle=\langle f,u_n-w\rangle .
\end{equation*}
Because of $w=b$ on $\Gamma_3$, the following inequality holds
\begin{eqnarray*}
&&\hspace{-0.6cm}
\langle Lu_n,u_n-w\rangle=\int_{\Gamma_3}l(x,u_n(x))(u_n(x)-w(x))\,d\Gamma
\\
&&
=\int_{\Gamma_3}l(x,u_n(x))(u_n(x)-b)\,d\Gamma
\ge \int_{\Gamma_3}l(x,b)(u_n(x)-b)\,d\Gamma =0.
\end{eqnarray*}
This implies
\begin{equation}\label{stineqe}
\langle Au_n+Bu_n,w-u_n\rangle
\ge \langle f,w-u_n\rangle.
\end{equation}
From the monotonicity of $A$ and $B$, we infer that
\begin{eqnarray*}
\langle Aw+Bw,w-u_n\rangle\ge \langle f,w-u_n\rangle.
\end{eqnarray*}
Passing to the limit as $n\to\infty$
in the inequality above, one gets
\begin{equation*}
\langle Aw+Bw,w-u\rangle\ge \langle f,w-u\rangle
\ \ \mbox{for all} \ \ w\in K.
\end{equation*}
Due to $u\in K$, for any $t\in (0,1)$ and $v\in K$, we have $w_t:=tv+(1-t)u\in K$.
Inserting $w=w_t$ into the inequality above,
it gives
\begin{equation*}
\langle Au+Bu,v-u\rangle=\lim_{t\to 0}\, \langle Aw_t+Bw_t,v-u\rangle\ge \langle f,v-u\rangle
\end{equation*}
for all $v\in K$, namely,
\begin{equation*}
\langle Au+Bu,v\rangle= \langle f,v\rangle
\end{equation*}
for all $v\in K_0$.
From assertion (i), we know that $u_\infty$ is the unique solution of Problem~\ref{problemss1}.
Therefore, we deduce that $u=u_\infty$. Since, every weakly convergent subsequence of $\{u_n\}$ converges weakly to the same limit $u_\infty$, it follows that  the whole sequence $\{u_n\}$ converges weakly to $u_\infty$.

Finally, it is easy to prove that $u_n$ converges strongly in $V$ to $u_\infty$.
Indeed, putting $w=u_\infty$ into (\ref{stineqe}),
passing to the lower limit as $n\to\infty$ in the resulting inequality, and taking into account the monotonicity of $B$, we obtain
\begin{equation*}
\limsup_{n\to\infty}\,
\langle Au_n,u_n-u_\infty\rangle
\le \limsup_{n\to\infty}\,
\langle f,u_n-u_\infty\rangle
+\limsup_{n\to\infty}\langle Bu_\infty,u_\infty-u_n\rangle=0.
\end{equation*}
This inequality combined with the $(S_+)$-property of  operator $A$ implies that $u_n\to u_\infty$ in $V$ as $n\to\infty$.
\end{proof}

\section{Optimal control and asymptotic analysis}\label{Section4}

In this section we investigate two optimal control problems driven by mixed boundary value problems,
Problem~\ref{problemss1} and Problem~\ref{problemss2}, respectively.
We prove existence of optimal controls
and establish a result on the asymptotic convergence
of optimal control-state pairs,
when the parameter $\alpha$ tends to infinity.

Let $H=L^{p'}(\Omega)$. Given a measured datum
$z_d\in L^p(\Omega)$ and two regularization parameters
$\lambda$, $\rho>0$,
we consider the following distributed optimal control problems governed by Problem~\ref{problemss1} and Problem~\ref{problemss2}, respectively.
\begin{Problem}\label{problemss3}
Find $g^*\in H$ such that
\begin{equation}\label{eqn4.1}
J(g^*)=\min_{g\in H}J(g),
\end{equation}
where the cost functional $J$ is defined by
\begin{equation}\label{eqn4.2}
J(g)=\frac{\lambda}{p}\,
\|u_g-z_d\|_{L^p(\Omega)}^p
+\frac{\rho}{p'}\, \|g\|_{L^{p'}(\Omega)}^{p'},
\end{equation}
and $u_g$ is the unique solution to Problem~$\ref{problemss1}$ corresponding
to $g\in L^{p'}(\Omega)$.
\end{Problem}
\noindent
and
\begin{Problem}\label{problemss4}
Given $\alpha>0$, find $g^*\in H$ such that
\begin{equation}\label{eqn4.3}
J_\alpha(g^*)=\min_{g\in H}J_\alpha(g),
\end{equation}
where the cost functional $J_\alpha$ is defined by
\begin{equation}\label{eqn4.4}
J_\alpha(g)=\frac{\lambda}{p}\,
\|u_{\alpha g}-z_d\|_{L^p(\Omega)}^p
+\frac{\rho}{p'}\, \|g\|_{L^{p'}(\Omega)}^{p'},
\end{equation}
and $u_{\alpha g}$ is the unique solution to Problem~$\ref{problemss2}$ corresponding to $g\in L^{p'}(\Omega)$ and $\alpha>0$.
\end{Problem}
A control-state pair $(g^*, u_{g^*})$ on which the infimum of (\ref{eqn4.1}) is attained is called
an optimal solution to Problem~\ref{problemss3}.
An analogous notion is applied
to Problem~\ref{problemss4}.

\smallskip

The first result of this section is on existence
of solutions to Problems~\ref{problemss3} and~\ref{problemss4}.
\begin{Theorem}\label{theorem2}
Assume that $H(l)$ and $r\in L^{p'}_+(\Gamma_2)$ hold. Then, we have
\begin{itemize}
\item[{\rm (i)}]
Problem~$\ref{problemss3}$ has at least one
optimal solution $(g^*,u^*_{g^*})\in H\times K$,
\item[{\rm (ii)}]
for each $\alpha>0$, Problem~$\ref{problemss4}$
has at least one optimal solution
$(g_\alpha^*,u^*_{\alpha g^*_\alpha})\in H\times V$.
\end{itemize}
\end{Theorem}

\begin{proof}
We prove statement (ii), while assertion (i) can be obtained in a similar way.
For any $\alpha>0$ fixed, it follows from
definition (\ref{eqn4.4}) that
$J_\alpha$ is bounded from below. This permits us to find a minimizing sequence $\{g_n\}\subset H$ of Problem~\ref{problemss4} such that
\begin{equation}\label{eqn4.5}
\lim_{n\to\infty} J_\alpha(g_n)=\inf_{g\in H} J_\alpha(g):=m_\alpha\ge 0.
\end{equation}
By the coercivity of $J_\alpha$, we can see that sequence $\{g_n\}$ is bounded in $L^{p'}(\Omega)$.
By the reflexivity of $L^{p'}(\Omega)$,
we may assume, passing to a subsequence if necessary,
that
\begin{equation}\label{eqn4.6}
g_n\wto g \ \ \mbox{in} \ \ L^{p'}(\Omega)
\end{equation}
for some $g\in H$.
Let us denote by $u_n\in V$ the unique solution to Problem~\ref{problemss2} corresponding to $g=g_n$ and $\alpha>0$.
We claim that $\{u_n\}$ is bounded in $V$. Let $\varepsilon :=\frac{1}{2 \alpha c_V^p}$.
For every $n\in\mathbb N$, a simple computation
gives (see (\ref{eqns3.9}), for example)
\begin{eqnarray}\label{estimatebound}
&&\hspace{-1.0cm}
\|g_n\|_{L^{p'}(\Omega)}\|u_n\|_{L^p(\Omega)}+\|r\|_{L^{p'}(\Gamma_2)}\|u_n\|_{L^p(\Gamma_2)}
\\[1mm]
&& \ge \langle f_n,u_n\rangle=\langle Au_n+Bu_n+\alpha Lu_n, u_n\rangle\nonumber\\[1mm]
&& \ge  \|u_n\|_{V}^p
+\mu\|\nabla u_n\|_{L^q(\Omega;\real^N)}^q
+\beta\|u_n\|_{L^\theta(\Omega)}^\theta-\alpha
b|\Gamma_3|^\frac{1}{p}\|a_l\|_{L^{p'}(\Gamma_3)} -\alpha b_lb|\Gamma_3| \nonumber
\\[1mm]
&&\qquad
-\varepsilon\, \alpha \, c_V^p\, \|u_n\|_V^p
-\alpha \, c(\varepsilon), \nonumber
\end{eqnarray}
where $f_n\in V^*$ is defined by
\begin{equation*}
\langle f_n,v\rangle=\int_\Omega g_n(x)v(x)\,dx-\int_{\Gamma_2}r(x)v(x)\,d\Gamma
\ \ \mbox{for all}\ \ v\in V.
\end{equation*}
The latter combined with the continuity
of the embeddings of $V$ to $L^p(\Omega)$
and of $V$ to $L^{p}(\Gamma_2)$ implies that sequence  $\{u_n\}$ is bounded in $V$. Without any loss of generality, we may suppose that
\begin{equation}\label{eqn4.7}
u_n\wto u \ \ \mbox{in} \ \ V \ \mbox{as} \ n\to\infty
\end{equation}
with some $u\in V$.

Next, we verify that $u$ is the unique solution of Problem~\ref{problemss2} corresponding to
$g\in L^{p'}(\Omega)$ and $\alpha>0$.
In fact, for each $n\in\mathbb N$, one has
\begin{equation}\label{eqn4.8}
\langle Au_n+Bu_n+\alpha Lu_n,w\rangle=\langle f_n,w\rangle
\end{equation}
for all $w\in V$. We insert $w=u-u_n$ into (\ref{eqn4.8}) to get
\begin{equation*}
\langle Au_n,u_n-u\rangle=\langle Bu_n+\alpha Lu_n-f_n,u-u_n\rangle.
\end{equation*}
Passing to the upper limit as $n\to\infty$ in this equality and using the compactness of embeddings
of $V$ to $L^p(\Omega)$ and of $V$ to $L^{p}(\Gamma_2)$,
and the monotonicity of $B$ and $L$, we obtain
\begin{equation*}
\limsup_{n\to\infty}\, \langle Au_n,u_n-u\rangle
\le \limsup_{n\to\infty}\, \langle Bu+\alpha Lu-f_n,u-u_n\rangle=0.
\end{equation*}
Taking into account the above result and the fact that $A$ satisfies $(S_+)$-property, we find that
$u_n\to u$ in $V$ as $n\to\infty$.
Letting $n\to\infty$ in equation (\ref{eqn4.8}),
one has
\begin{equation*}
\langle Au+Bu+\alpha Lu,w\rangle=\langle f,w\rangle
\end{equation*}
for all $w\in V$. Now, it is obvious that $u$
is the unique solution of Problem~\ref{problemss2} corresponding to $g\in L^{p'}(\Omega)$ and $\alpha>0$.

Finally, from the weak lower semicontinuity of
the norm function $g\mapsto \|g\|_{L^{p'}(\Omega)}$,
we infer that
\begin{eqnarray*}
&&\hspace{-1.0cm}
\liminf_{n\to\infty}J_{\alpha}(g_n)
=
\liminf_{n\to\infty}
\left(\frac{\lambda}{p}\,
\|u_{n}-z_d\|_{L^p(\Omega)}^p
+\frac{\rho}{p'} \, \|g_n\|_{L^{p'}(\Omega)}^{p'}\right)
\\
&&
= \lim_{n\to\infty}\frac{\lambda}{p} \,
\|u_{n}-z_d\|_{L^p(\Omega)}^p
+\liminf_{n\to\infty}\, \frac{\rho}{p'}\,
\|g_n\|_{L^{p'}(\Omega)}^{p'}\\
&& \ge \frac{\lambda}{p}\, \|u-z_d\|_{L^p(\Omega)}^p
+\frac{\rho}{p'}\, \|g\|_{L^{p'}(\Omega)}^{p'}
= J_{\alpha}(g).
\end{eqnarray*}
This together with (\ref{eqn4.5}) entails that
$(g,u)\in H\times V$ is an optimal solution to  Problem~\ref{problemss4}.
This completes the proof.
\end{proof}

The second result of this section is on the asymptotic behavior of the optimal solutions to Problem~\ref{problemss4}.
\begin{Theorem}\label{theorem3}
Assume that $H(l)$ and $r\in L^{p'}_+(\Gamma_2)$ hold.
Let $\{\alpha_n\}$ be a sequence
such that $\alpha_n>0$ and $\alpha_n\to +\infty$ as $n\to\infty$, and  let $(g_{\alpha_n},u_{\alpha_ng_{\alpha_n}})$
be an optimal solution for Problem~\ref{problemss4}. Then,
there exist an optimal solution
$(g_\infty^*,u_{\infty g_\infty^*}^*)$ for Problem~\ref{problemss3} and a subsequence of $\{(g_{\alpha_n},u_{\alpha_ng_{\alpha_n}})\}$, still denoted by the same way, such that
\begin{equation*}
g_{\alpha_n}\to g_\infty^*
\ \ \mbox{\rm in} \ \ L^{p'}(\Omega) \ \ \mbox{\rm and} \ \  u_{\alpha_ng_{\alpha_n}}\to u_{\infty g_\infty^*}^*
\ \ \mbox{\rm in} \ \ V \ \mbox{\rm as} \ n\to\infty.
\end{equation*}
Moreover, the sequence $\{ J_{\alpha_n}(g_{\alpha_n}) \}$
of optimal values for Problem~\ref{problemss4}
converges to the optimal value $J(g_\infty^*)$
of Problem~\ref{problemss3}.

If Problem~\ref{problemss3} has a unique optimal solution, then the whole sequence
$\{(g_{\alpha_n},u_{\alpha_ng_{\alpha_n}})\}$
converges in $L^{p'}(\Omega) \times V$ to
$(g_\infty^*,u_{\infty g_\infty^*}^*)$
as $n \to \infty$.
\end{Theorem}

\begin{proof}
Let $\{\alpha_n\}$ be a sequence such that
$\alpha_n>0$ and $\alpha_n\to +\infty$ as $n\to\infty$, and let $(g_n,u_n):=(g_{\alpha_n},u_{\alpha_ng_{\alpha_n}})$
be an optimal solution to Problem~\ref{problemss4} corresponding to $\alpha_n>0$.
It obvious that
\begin{equation*}
\frac{\lambda}{p}\|u_n-z_d\|_{L^p(\Omega)}^p+\frac{\rho}{p'}\|g_n\|_{L^{p'}(\Omega)}^{p'}=J_{\alpha_n}(g_n)\le J_{\alpha_n}(g)=\frac{\lambda}{p}\|\widetilde u_n-z_d\|_{L^p(\Omega)}^p+\frac{\rho}{p'}\|g\|_{L^{p'}(\Omega)}^{p'},
\end{equation*}
for any $g \in H$,
where $\widetilde u_n\in V$ is the unique solution of Problem~\ref{problemss2} corresponding to $\alpha_n>0$ and $g\in L^{p'}(\Omega)$.
Employing Theorem~\ref{theorem1}(vii), we have $\widetilde u_n\to u_\infty$ in $V$ as $n\to\infty$, where $u_\infty$ is the solution of Problem~\ref{problemss1} associated with $g\in L^{p'}(\Omega)$. We can observe that sequence
$\{g_n\}$ is bounded in $L^{p'}(\Omega)$.
Without any loss of generality, we may assume that
\begin{equation*}
g_n\wto g \ \ \mbox{in} \ \ L^{p'}(\Omega)
\ \mbox{as} \ n\to\infty
\end{equation*}
for some $g\in H$.
Next, for each $n\in \mathbb N$, we have
\begin{equation*}
\langle Au_n+Bu_n,u_{\infty}-u_n\rangle+\alpha_n\langle Lu_n,u_{\infty}-u_n\rangle=\langle f_n,u_{\infty}-u_n\rangle.
\end{equation*}
Using the same arguments as in the proof of (\ref{estimatebound}), we obtain that sequence  $\{u_n\}$ is bounded in $V$.
Passing to a subsequence if necessary,
we may suppose that
\begin{equation*}
u_n\wto u \ \ \mbox{in} \ \ V \ \mbox{as} \ n\to\infty
\end{equation*}
for some $u\in V$. Since $\{u_n\}$ is bounded in $V$,
by the former argument, one has
\begin{equation}\label{eqn4.9}
\langle Lu_n,u_n-u_{\infty}\rangle\le \frac{1}{\alpha_n}\left[\langle Au_n+Bu_n-f_n,u_{\infty}-u_n\rangle\right]\le \frac{c_4}{\alpha_n}
\end{equation}
for some $c_4>0$ which is independent of $n$.
We use the compactness of embedding of $V$ to $L^{p}(\Gamma_3)$, apply (\ref{eqn4.9})
and the Lebesgue dominated convergence theorem to get
\begin{equation*}
0=\lim_{n\to\infty}\frac{c_4}{\alpha_n}\ge \lim_{n\to\infty}\langle Lu_n,u_n-u_{\infty }\rangle
=\langle Lu,u-u_{\infty }\rangle\ge \langle Lu_{\infty },u-u_{\infty }\rangle=0.
\end{equation*}
Hence, it holds $u(x)=b$ for a.e. $x\in \Gamma_3$,
i.e., $u\in K$.
We take the limit as $n\to\infty$ in the following inequality
\begin{equation}\label{Spineq}
\langle Au_n+Bu_n,w-u_n\rangle\ge \langle f_n,w-u_n\rangle
\end{equation}
for all $w\in K$ to get
$\langle Au + Bu, v \rangle= \langle f, v \rangle$
for all $v\in K_0$. This means that $u\in K$ is the unique solution of Problem~\ref{problemss1} corresponding to $g$.
Choosing $w=u$ into (\ref{Spineq}) and passing to the upper limit as $n\to\infty$, we use the $(S_+)$-property of $A$ to get that $u_n\to u$ in $V$ as $n\to\infty$.
Note that from
$$
J_{\alpha_n}(g_n)\le J_{\alpha_n}(h)
\ \ \mbox{for all} \ \ h\in H,
$$
we have
\begin{eqnarray*}
&&\hspace{-0.5cm}
J(g)=\frac{\lambda}{p}\|u-z_d\|_{L^p(\Omega)}^p
+\frac{\rho}{p'}\|g\|_{L^{p'}(\Omega)}^{p'}
\\
&& \le \liminf_{n\to\infty}\Big( \frac{\lambda}{p}
\|u_n-z_d\|_{L^p(\Omega)}^p
+\frac{\rho}{p'}\|g_n\|_{L^{p'}(\Omega)}^{p'}\Big)
=\liminf_{n\to\infty} J_{\alpha_n}(g_n)
\le \liminf_{n\to\infty} J_{\alpha_n}(h)
\\
&& =\liminf_{n\to\infty}\Big(
\frac{\lambda}{p}\| \widehat {u}_n-z_d\|_{L^p(\Omega)}^p
+\frac{\rho}{p'}\|h\|_{L^{p'}(\Omega)}^{p'}\Big)
=
\frac{\lambda}{p}\| \widehat u-z_d\|_{L^p(\Omega)}^p
+\frac{\rho}{p'}\|h\|_{L^{p'}(\Omega)}^{p'}
= J(h)
\end{eqnarray*}
for all $h\in H$, where we have applied Theorem~\ref{theorem1}(vii) and $\widehat u_n$ is the unique solution of Problem~\ref{problemss2} corresponding to $h\in L^{p'}(\Omega)$ and $\alpha_n>0$. Therefore, we can see that $g$ is also a solution of Problem~\ref{problemss3}. In the meanwhile, we have $g=g_\infty^*$ and $u=u_{\infty g_\infty^*}^*$.

Finally, we show that $g_n$ converges strongly to $g_\infty^*$ in $L^{p'}(\Omega)$.
Keeping in mind that
\begin{equation*}
\|g_\infty^*\|_{L^{p'}(\Omega)}\le \liminf_{n\to\infty}\|g_n\|_{L^{p'}(\Omega)},
\end{equation*}
and
$J_{\alpha_n}(g_n)\le J_{\alpha_n}(g_\infty^*)$,
we have
\begin{equation}\label{eqn4.11}
J(g_\infty^*)\le \liminf_{n\to\infty}J_{\alpha_n}(g_n)
\le \limsup_{n\to\infty}J_{\alpha_n}(g_n)\le \limsup_{n\to\infty}J_{\alpha_n}(g_\infty^*)
=J(g_\infty^*).
\end{equation}
Hence
\begin{equation}\label{NEW1}
J(g_\infty^*)
= \lim_{n\to\infty}J_{\alpha_n}(g_n)
\ \ \mbox{and} \ \
\|g_\infty^*\|_{L^{p'}(\Omega)}= \lim_{n\to\infty}\|g_n\|_{L^{p'}(\Omega)}.
\end{equation}
Thus, we conclude that $g_n\to g_\infty^*$  in $L^{p'}(\Omega)$ by using the triangle inequality and the fact that $g_n\wto g$ in $L^{p'}(\Omega)$.
The convergence of the sequence
$\{ J_{\alpha_n}(g_{\alpha_n}) \}$
of optimal values for Problem~\ref{problemss4}
to the optimal value $J(g_\infty^*)$
of Problem~\ref{problemss3} is a consequence of (\ref{NEW1}).
This completes the proof.
\end{proof}

\begin{Remark}
{\rm
Let $y_d\in L^p(\Omega;\real^N)$ be a desired element. Theorems~\ref{theorem2} and~\ref{theorem3}
established in this section are valid when the cost functionals (\ref{eqn4.2}) and (\ref{eqn4.4}) are replaced by the ones
\begin{eqnarray*}
&&
J(g)=\frac{\lambda}{p}\,
\|\nabla u_g-y_d\|_{L^p(\Omega;\real^N)}^p
+\frac{\rho}{p'}\, \|g\|_{L^{p'}(\Omega)}^{p'}, \\[1mm]
&&
J_\alpha(g)=\frac{\lambda}{p}\,
\|\nabla u_{\alpha g}-y_d\|_{L^p(\Omega;\real^N)}^p
+\frac{\rho}{p'}\, \|g\|_{L^{p'}(\Omega)}^{p'},
\end{eqnarray*}
respectively.
}
\end{Remark}

\section*{Acknowledgements}
The authors wish to thank the two knowledgeable referees for their useful remarks in order to improve the paper.

This project has received funding from the NNSF of China Grant Nos. 12001478, 12026255 and 12026256,
the European Union's Horizon 2020 Research and Innovation Programme under the Marie Sk{\l}odowska-Curie grant agreement No. 823731 CONMECH,
National Science Center of Poland under Project No.
2021/41/B/ST1/01636,
and the Startup Project of Doctor Scientific Research of Yulin Normal University No. G2020ZK07.
It is also supported by Natural Science Foundation of Guangxi Grant Nos.  2021GXNSFFA196004, 2020GXNSFBA297137 and 2018GXNSFAA281353, and the Ministry of Science and Higher Education of Republic of Poland under Grants Nos. 4004/GGPJII/H2020/2018/0 and 440328/PnH2/2019.

\section*{Data availability statements}

Data sharing not applicable to this article as no data sets were generated or analysed during the current study.

\end{document}